\documentclass[12pt,a4paper,twoside]{article}

\usepackage{a4wide, amssymb, amsmath, amsthm, graphics, comment, xspace, enumerate}
\usepackage{graphicx}
\usepackage[a4paper,colorlinks=true,citecolor=blue,urlcolor=blue,linkcolor=blue,bookmarksopen=true,unicode=true,
          pdffitwindow=true]{hyperref}
\usepackage{algorithmic, algorithm}

\allowdisplaybreaks

\newtheorem{theorem}{Theorem}[section]

\newtheorem{lemma}[theorem]{Lemma}

\newtheorem{remark}[theorem]{Remark}

\newcommand{\NN}{\mathbb N}
\newcommand{\CC}{\mathbb C}
\newcommand{\RR}{\mathbb R}
\newcommand{\ZZ}{\mathbb Z}
\newcommand{\ds}{\displaystyle}
\newcommand{\EE}{\mathcal E}
\newcommand{\DD}{\mathcal D}
\newcommand{\SSS}{\mathcal S}

\newcommand{\beq}{\begin{eqnarray}}
\newcommand{\eeq}{\end{eqnarray}}
\newcommand{\beqs}{\begin{eqnarray*}}
\newcommand{\eeqs}{\end{eqnarray*}}
\begin{document}

\title{Laplace transform in spaces of ultradistributions}

\author{Bojan Prangoski}

\date{}
\maketitle

\begin{abstract}
The Laplace transform in Komatsu ultradistributions is considered. Also, conditions are given under which an analytic function is a Laplace transformation of an ultradistribution.
\end{abstract}


\section{Introduction}

The Laplace transform of distributions was defined and studied by Schwartz, \cite{SchwartzK}. Later, Carmichael and Pilipovi\'c in \cite{CP} (see also \cite{PilipovicK}), considered the Laplace transform in $\Sigma'_{\alpha}$ of Beurling-Gevrey tempered ultradistributions and obtained some results concerning the so-called tempered convolution. In particular, they gave a characterization of the space of Laplace transforms of elements from $\Sigma'_{\alpha}$ supported by an acute closed cone in $\RR^d$. Komatsu has given a great contribution to the investigations of the Laplace transform in ultradistribution and hyperfunction spaces considering them over appropriate domains, see \cite{kl} and references therein  (see also \cite{zh}). Michalik in \cite{Mic} and Lee and Kim in \cite{Kim} have adapted the space of ultradistribution and Fourier hyperfunctions to the definition of the Laplace transform, following ideas of Komatsu. Our approach is different. We develop the theory within the space of already constructed ultradistributions of Beurling and Roumieu type. The ideas in the proofs of the two main theorems (theorem \ref{t1} and theorem \ref{t2}) are similar to those in \cite{Vladimirov} in the case of Schwartz distributions. In these theorems are characterized ultradistributions defined on the whole $\RR^d$ through the estimates of their Laplace transforms. This is the main point of our investigations contrary to other authors who investigated generalized functions supported by cones. We consider a restricted class of ultradistributions assuming conditions $(M.1), (M.2)$ and $(M.3)$ (for example, cases $M_p=p!^s$, $s>1$) in order to obtain fine representations through the analysis of the corresponding class of subexponentially bounded entire functions. With weaker conditions, $(M.3)'$ instead of $(M.3),$ or even in the case of quasianalyticity, we can obtain different, technically more complicate, structural representations.

\section{Preliminaries}

The sets of natural, integer, positive integer, real and complex numbers are denoted by $\NN$, $\ZZ$, $\ZZ_+$, $\RR$, $\CC$. We use the symbols for $x\in \RR^d$: $\langle x\rangle =(1+|x|^2)^{1/2} $,
$D^{\alpha}= D_1^{\alpha_1}\ldots D_n^{\alpha_d},\quad D_j^
{\alpha_j}={i^{-1}}\partial^{\alpha_j}/{\partial x}^{\alpha_j}$, $\alpha=(\alpha_1,\alpha_2,\ldots,\alpha_d)\in\NN^d$. If $z\in\CC^d$, by $z^2$ we will denote $z^2_1+...+z^2_d$. Note that, if $x\in\RR^d$, $x^2=|x|^2$.\\
\indent Following \cite{Komatsu1}, we denote by  $M_{p}$    a  sequence  of  positive  numbers $M_0=1$ so that:\\
\indent $(M.1)$ $M_{p}^{2} \leq M_{p-1} M_{p+1}, \; \; p \in\ZZ_+$;\\
\indent $(M.2)$ $\ds M_{p} \leq c_0H^{p} \min_{0\leq q\leq p} \{M_{p-q} M_{q}\}$, $p,q\in \NN$, for some $c_0,H\geq1$;\\
\indent $(M.3)$  $\ds\sum^{\infty}_{p=q+1}   \frac{M_{p-1}}{M_{p}}\leq c_0q \frac{M_{q}}{M_{q+1}}$, $q\in \ZZ_+$,\\
although in some assertions we could assume the weaker ones $(M.2)'$ and $(M.3)'$ (see \cite{Komatsu1}). For a multi-index $\alpha\in\NN^d$, $M_{\alpha}$ will mean $M_{|\alpha|}$, $|\alpha|=\alpha_1+...+\alpha_d$. Recall,  $m_p=M_p/M_{p-1}$, $p\in\ZZ_+$ and the associated function for the sequence $M_{p}$ is defined by
\beqs
M(\rho)=\sup  _{p\in\NN}\log_+   \frac{\rho^{p}}{M_{p}} , \; \; \rho > 0.
\eeqs
It is non-negative, continuous, monotonically increasing function, which vanishes for sufficiently small $\rho>0$ and increases more rapidly then $(\ln \rho)^p$ when $\rho$ tends to infinity, for any $p\in\NN$.\\
\indent Let $U\subseteq\RR^d$ be an open set and $K\subset\subset U$ (we will use always this notation for a compact subset of an open set). Then $\EE^{\{M_p\},h}(K)$ is the space of all $\varphi\in \mathcal{C}^{\infty}(U)$ which satisfy $\ds\sup_{\alpha\in\NN^d}\sup_{x\in K}\frac{|D^{\alpha}\varphi(x)|}{h^{\alpha}M_{\alpha}}<\infty$ and $\DD^{\{M_p\},h}_K$ is the space of all $\varphi\in \mathcal{C}^{\infty}\left(\RR^d\right)$ with supports in $K$, which satisfy $\ds\sup_{\alpha\in\NN^d}\sup_{x\in K}\frac{|D^{\alpha}\varphi(x)|}{h^{\alpha}M_{\alpha}}<\infty$;
$$
\EE^{(M_p)}(U)=\lim_{\substack{\longleftarrow\\ K\subset\subset U}}\lim_{\substack{\longleftarrow\\ h\rightarrow 0}} \EE^{\{M_p\},h}(K),\,\,\,\,
\EE^{\{M_p\}}(U)=\lim_{\substack{\longleftarrow\\ K\subset\subset U}}
\lim_{\substack{\longrightarrow\\ h\rightarrow \infty}} \EE^{\{M_p\},h}(K),
$$
\beqs
\DD^{(M_p)}_K=\lim_{\substack{\longleftarrow\\ h\rightarrow 0}} \DD^{\{M_p\},h}_K,\,\,\,\, \DD^{(M_p)}(U)=\lim_{\substack{\longrightarrow\\ K\subset\subset U}}\DD^{(M_p)}_K,\\
\DD^{\{M_p\}}_K=\lim_{\substack{\longrightarrow\\ h\rightarrow \infty}} \DD^{\{M_p\},h}_K,\,\,\,\, \DD^{\{M_p\}}(U)=\lim_{\substack{\longrightarrow\\ K\subset\subset U}}\DD^{\{M_p\}}_K.
\eeqs
The spaces of ultradistributions and ultradistributions with compact support of Beurling and Roumieu type are defined as the strong duals of $\DD^{(M_p)}(U)$ and $\EE^{(M_p)}(U)$, resp. $\DD^{\{M_p\}}(U)$ and $\EE^{\{M_p\}}(U)$. For the properties of these spaces, we refer to \cite{Komatsu1}, \cite{Komatsu2} and \cite{Komatsu3}. In the future we will not emphasize the set $U$ when $U=\RR^d$. Also, the common notation for the symbols $(M_{p})$ and $\{M_{p}\} $ will be *.\\
\indent If $f\in L^{1} $, then its Fourier transform is defined by
$
(\mathcal{F}f)(\xi ) = \hat{f} (\xi) = \int_{{\RR^d}} e^{-ix\xi}f(x)dx, \; \; \xi \in {\RR^d}.
$

By $\mathfrak{R}$ is denoted a set of positive sequences which monotonically increases to infinity. For $(r_p)\in\mathfrak{R}$, consider the sequence $N_0=1$, $N_p=M_p\prod_{j=1}^{p}r_j$, $p\in\ZZ_+$. One easily sees that this sequence satisfies $(M.1)$ and $(M.3)'$ and its associated function will be denoted by $N_{r_p}(\rho)$, i.e. $\ds N_{r_{p}}(\rho )=\sup_{p\in\NN} \log_+ \frac{\rho^{p }}{M_p\prod_{j=1}^{p}r_j}$, $\rho > 0$. Note, for given $r_{p}$ and every $k > 0 $ there is $\rho _{0} > 0$ such that $\ds N_{r_{p}} (\rho ) \leq M(k \rho )$, for $\rho > \rho _{0}$.\\
\indent It is said that $\ds P(\xi ) =\sum _{\alpha \in \NN^d}c_{\alpha } \xi^{\alpha}$, $\xi \in \RR^d$, is an ultrapolynomial of the class $(M_{p})$, resp. $\{M_{p}\}$, whenever the coefficients $c_{\alpha }$ satisfy the estimate $|c_{\alpha }|  \leq C L^{\alpha }M_{\alpha}$, $\alpha \in \NN^d$ for some $L > 0$ and $C>0$, resp. for every $L > 0 $ and some $C_{L} > 0$. The corresponding operator  $P(D)=\sum_{\alpha} c_{\alpha}D^{\alpha}$ is an ultradifferential operator of the class $(M_{p})$, resp. $\{M_{p}\}$ and they act continuously on $\EE^{(M_p)}(U)$ and $\DD^{(M_p)}(U)$, resp. $\EE^{\{M_p\}}(U)$ and $\DD^{\{M_p\}}(U)$ and the corresponding spaces of ultradistributions.\\
\indent We denote by $\SSS^{M_{p},m}_{2} \left(\RR^d\right)$, $m > 0$, the space of all smooth functions $\varphi$ which satisfy
\beq\label{75}
\sigma_{m,2}(\varphi ): = \left( \sum_{\alpha,\beta\in\NN^d} \int_{\RR^d} \left|\frac{m^{|\alpha|+|\beta|}\langle x\rangle^{|\alpha|}D^{\beta}\varphi(x)}{M_{\alpha}M_{\beta}}\right| ^{2} dx \right) ^{1/2}<\infty,
\eeq
supplied with the topology induced by the norm $\sigma _{m,2}$. The spaces $\SSS'^{(M_{p})}$ and $\SSS'^{\{M_{p}\}}$ of tempered ultradistributions of Beurling and Roumieu type respectively, are defined as the strong duals of the spaces $\ds\SSS^{(M_{p})}=\lim_{\substack{\longleftarrow\\ m\rightarrow\infty}}\SSS^{M_{p},m}_{2}\left(\RR^d\right)$ and $\ds\SSS^{\{M_{p}\}}=\lim_{\substack{\longrightarrow\\ m\rightarrow 0}}\SSS^{M_{p},m}_{2}\left(\RR^d\right)$, respectively. All the good properties of $\SSS^*$ and its strong dual follow from the equivalence of the sequence of norms $\sigma_{m,2}$, $m > 0$, with each of the following sequences of norms (see \cite{PilipovicK}, \cite{PilipovicU}):\\
\indent $(a)$ $\sigma_{m,p}$, $m > 0$; $p \in [1, \infty ]$ is fixed;\\
\indent $(b)$ $s_{m,p}$, $m > 0$; $p \in [1,\infty ]$ is fixed, where $\ds s_{m,p}(\varphi): =\sum_{\alpha ,\beta  \in  \NN^d}\frac{m^{|\alpha| +|\beta| }\| |\cdot|^{\beta }D^{\alpha}\varphi(\cdot)\|_{L^p}}{M_{\alpha }M_{\beta }}$;\\
\indent $(c)$ $s_{m}$, $m > 0$, where $\ds s_{m}(\varphi):=\sup_{\alpha\in \NN^d}\frac{m^{|\alpha|}
\| D^{\alpha}\varphi(\cdot) e^{M(m|\cdot|)}\|_{L_{\infty}}}{M_{\alpha }}$.\\
If we denote by $\SSS^{M_p,m}_{\infty}\left(\RR^d\right)$ the space of all infinitely differentiable functions on $\RR^d$ for which the norm $\sigma_{m,\infty}$ is finite (obviously it is a Banach space), then $\ds\SSS^{(M_p)}\left(\RR^d\right)=\lim_{\substack{\longleftarrow\\ m\rightarrow\infty}} \SSS^{M_p,m}_{\infty}\left(\RR^d\right)$ and $\ds\SSS^{\{M_p\}}\left(\RR^d\right)=\lim_{\substack{\longrightarrow\\ m\rightarrow 0}} \SSS^{M_p,m}_{\infty}\left(\RR^d\right)$. Also, for $m_2>m_1$, the inclusion $\SSS^{M_p,m_2}_{\infty}\left(\RR^d\right)\longrightarrow\SSS^{M_p,m_1}_{\infty}\left(\RR^d\right)$ is a compact mapping. In \cite{PilipovicT} and \cite{PilipovicK} it is proved that $\ds\SSS^{\{M_{p}\}} = \lim_{\substack{\longleftarrow\\ r_{i}, s_{j} \in \mathfrak{R}}}\SSS^{M_{p}}_{(r_{p}),(s_{q})}$, where $\ds\SSS^{M_{p}}_{(r_{p}),(s_{q})}=\left\{\varphi \in \mathcal{C}^{\infty} \left(\RR^d\right)|\gamma _{(r_{p}),(s_{q})}(\varphi)<\infty\right\}$ and
\beqs
\gamma_{(r_{p}),(s_{q})}(\varphi) =\sup_{\alpha,\beta\in  \NN^d}\frac{\left\|\langle x\rangle^{|\beta|}D^{\alpha}\varphi(x)\right\|_{L^{2}}} {\left(\prod^{|\alpha|}_{p=1}r_{p}\right)M_{\alpha}\left(\prod^{|\beta|}_{q=1}s_{q}\right)M_{\beta}}.
\eeqs
Also, the Fourier transform is a topological automorphism of $\SSS^*$ and of $\SSS'^*$.

\section{Laplace transform}

For a set $B\subseteq\RR^d$ denote by $\mathrm{ch\,}B$ the convex hull of $B$.

\begin{theorem}\label{t1}
Let $B$ be a connected open set in $\RR^d_{\xi}$ and $T\in\DD'^{*}(\RR^d_x)$ be such that, for all $\xi\in B$, $e^{-x\xi}T(x)\in\SSS'^{*}(\RR^d_x)$. Then the Fourier transform $\mathcal{F}_{x\rightarrow\eta}\left(e^{-x\xi}T(x)\right)$ is an analytic function of $\zeta=\xi+i\eta$ for $\xi\in \mathrm{ch\,}B$, $\eta\in\RR^d$. Furthermore, it satisfies the following estimates:\\
\indent for every $K\subset\subset\mathrm{ch\,}B$ there exist $k>0$ and $C>0$, resp. for every $k>0$ there exists $C>0$, such that
\beq\label{3}
|\mathcal{F}_{x\rightarrow\eta}(e^{-x\xi}T(x))(\xi+i\eta)|\leq Ce^{M(k|\eta|)},\, \forall \xi\in K, \forall\eta\in\RR^d.
\eeq
\end{theorem}

\begin{proof} Let $K$ be a fixed compact subset of $\mathrm{ch\,}B$. There exists $0<\varepsilon<1/4$ and $\xi^{(1)},...,\xi^{(l)}\in B$ such that the convex hull $\Pi$ of the set $\{\xi^{(1)},...,\xi^{(l)}\}$ contains the closed $4\varepsilon$ neighborhood of $K$ (obviously $\Pi\subset\subset \mathrm{ch\,}B$). We shell prove that the set
\beq\label{5}
\left\{S\in\DD'^{*}|S(x)=T(x)e^{-x\xi+\varepsilon\sqrt{1+|x|^2}},\xi\in K\right\}
\eeq
is bounded in $\SSS'^{*}$. Note that by the condition in the theorem $T(x)e^{-x\xi}\in\SSS'^{*}$ and $e^{\varepsilon\sqrt{1+|x|^2}}$ is the restriction on the real axis of the function $e^{\varepsilon\sqrt{1+z^2}}$ that is analytic and single valued on the strip $\RR^d+i\{y\in\RR^d||y|<1/4\}$, and hence $e^{\varepsilon\sqrt{1+|x|^2}}$ is in $\EE^{*}$. Note that
\beq\label{7}
T(x)e^{-x\xi+\varepsilon\sqrt{1+|x|^2}}=\sum_{k=1}^l e^{\varepsilon\sqrt{1+|x|^2}}a(x,\xi)T(x)e^{-x\xi^{(k)}},
\eeq
where $\ds a(x,\xi)=e^{-x\xi}\left(\sum_{k=1}^l e^{-x\xi^{(k)}}\right)^{-1}$. The function $a(x,\xi)$ satisfies the following conditions:\\
\indent $i)$ $0<a(x,\xi)\leq 1$, $(x,\xi)\in\RR^d\times\Pi$;\\
\indent $ii)$ $e^{\varepsilon'\sqrt{1+|x|^2}}a(x,\xi)\leq e^{\varepsilon'}$, $(x,\xi)\in\RR^d\times K$, and $\forall\varepsilon'\leq 4\varepsilon$;\\
\indent $iii)$ $a(x,\xi)\in \mathcal{C}^{\infty}\left(\RR^{2d}\right)$.\\
$iii)$ it's obvious. To prove $i)$, take $\xi\in\Pi$. Then there exist $t_1,...,t_l\geq0$ such that
$\ds\xi=\sum_{k=1}^l t_k\xi^{(k)}$ and $\ds\sum_{k=1}^l t_k=1$. Then, by the weighted arithmetic mean-geometric mean inequality, we have
\beqs
e^{-x\xi}=\prod_{k=1}^l e^{-xt_k\xi^{(k)}}\leq\sum_{k=1}^l t_ke^{-x\xi^{(k)}}\leq\sum_{k=1}^l e^{-x\xi^{(k)}},
\eeqs
from where it follows $i)$. For the prove of $ii)$, note that, for $(x,\xi)\in\RR^d\times K$,
\beqs
e^{\varepsilon'\sqrt{1+|x|^2}}a(x,\xi)\leq e^{\varepsilon'+\varepsilon'|x|}a(x,\xi)=e^{\varepsilon'}\max_{|t|\leq\varepsilon'}e^{-tx}a(x,\xi)=
e^{\varepsilon'}\max_{|t|\leq\varepsilon'}a(x,\xi+t)\leq e^{\varepsilon'},
\eeqs
where the last inequality follows from $i)$.\\
\indent Now we will estimate the derivatives of $a(x,\xi)$. Let $\ds s=\max_{\xi\in\Pi} |\xi|$. Then $a(z,\xi)$ is an analytic function of $z=x+iy$ on the strip $\RR^d+i\{y\in\RR^d||y|s<\pi/4\}$, for every fixed $\xi\in\Pi$, because
\beqs
\left|\sum_{k=1}^l e^{-z\xi^{(k)}}\right|^2=\left|\sum_{k=1}^l e^{-x\xi^{(k)}}e^{-iy\xi^{(k)}}\right|^2\geq
\left(\sum_{k=1}^l e^{-x\xi^{(k)}}\cos y\xi^{(k)}\right)^2\geq\left(\sum_{k=1}^l e^{-x\xi^{(k)}}\frac{\sqrt{2}}{2}\right)^2,
\eeqs
and hence
\beq\label{10}
\left|\sum_{k=1}^l e^{-z\xi^{(k)}}\right|\geq\frac{\sqrt{2}}{2}\sum_{k=1}^l e^{-x\xi^{(k)}}>0,
\eeq
Take $0<r<1/\sqrt{d}$ so small such that $rs\sqrt{d}<\pi/4$. Then, from Cauchy integral formula, we have
\beqs
|\partial_z^{\alpha}a(x,\xi)|\leq \frac{\alpha !}{r^{|\alpha|}}
\sup_{|w_1-x_1|\leq r,...,|w_d-x_d|\leq r}\left|\frac{e^{-w\xi}}{\sum_{k=1}^l e^{-w\xi^{(k)}}}\right|.
\eeqs
If we use the inequality (\ref{10}), we get (we put $w=u+iv$)
\beqs
\left|\frac{e^{-(u+iv)\xi}}{\sum_{k=1}^l e^{-(u+iv)\xi^{(k)}}}\right|&\leq& \frac{\sqrt{2}e^{-u\xi}}{\sum_{k=1}^l e^{-u\xi^{(k)}}}
=\frac{\sqrt{2}e^{-x\xi}e^{-(u-x)\xi}}{\sum_{k=1}^l e^{-x\xi^{(k)}}e^{-(u-x)\xi^{(k)}}}\\
&\leq&\frac{\sqrt{2}e^{-x\xi}e^{|u-x||\xi|}}{\sum_{k=1}^l e^{-x\xi^{(k)}}e^{-|u-x|\left|\xi^{(k)}\right|}}
\leq\frac{\sqrt{2}e^{-x\xi}e^{rs\sqrt{d}}}
{\sum_{k=1}^l e^{-x\xi^{(k)}}e^{-rs\sqrt{d}}}=\sqrt{2}e^{2rs\sqrt{d}}a(x,\xi).
\eeqs
So, we obtain the estimate
\beq\label{12}
\left|\partial_x^{\alpha}a(x,\xi)\right|\leq \sqrt{2}e^{2s}\frac{\alpha !}{r^{|\alpha|}}a(x,\xi).
\eeq
Note that, by the previous estimate and the property $ii)$ of $a(x,\xi)$, it follows that $a(x,\xi)\in \SSS^{*}$ for every $\xi\in K$ and the set $\{a(x,\xi)|\xi\in K\}$ is a bounded set in $\SSS^{*}$. We will estimate the derivatives of $e^{\varepsilon\sqrt{1+|x|^2}}$. The function $e^{\varepsilon\sqrt{1+z^2}}$ is analytic on the strip $\RR^d+i\{y\in\RR^d||y|<1/4\}$, where we take the principal branch of the square root which is single valued and analytic on $\CC\backslash (-\infty,0]$. If we take $r<1/(8d)$, from the Cauchy integral formula, we get the estimate $\ds\left|\partial_z^{\alpha} e^{\varepsilon\sqrt{1+|x|^2}}\right|\leq\frac{\alpha !}{r^{|\alpha|}}
\sup_{|w_1-x_1|\leq r,...,|w_d-x_d|\leq r}\left|e^{\varepsilon\sqrt{1+w^2}}\right|$. Put $w=u+iv$ and estimate as follows
\beqs
\left|e^{\varepsilon\sqrt{1+w^2}}\right|&=& e^{\mathrm{Re\,}\left(\varepsilon\sqrt{1+w^2}\right)}\leq e^{\left|\varepsilon\sqrt{1+w^2}\right|}\leq e^{\varepsilon{\sqrt[4]{(1+|u|^2-|v|^2)^2+4(uv)^2}}}\leq
e^{\varepsilon{\sqrt{1+|u|^2-|v|^2+2|uv|}}}\\
&\leq& e^{\varepsilon{\sqrt{1+2|u|^2}}}\leq e^{\varepsilon{\sqrt{1+4|u-x|^2+4|x|^2}}}\leq e^{\varepsilon{\sqrt{1+1+4|x|^2}}}\leq e^{2\varepsilon{\sqrt{1+|x|^2}}}.
\eeqs
Hence
\beq\label{13}
\left|\partial_x^{\alpha} e^{\varepsilon\sqrt{1+|x|^2}}\right|\leq\frac{\alpha !}{r^{|\alpha|}}
e^{2\varepsilon\sqrt{1+|x|^2}}.
\eeq
If we take $r$ small enough we can make the previous estimates for the derivatives of $a(x,\xi)$ and $e^{\varepsilon\sqrt{1+|x|^2}}$ to hold for the same $r$. Now we obtain
\beqs
\left|D^{\alpha}_x \left(e^{\varepsilon\sqrt{1+|x|^2}}a(x,\xi)\right)\right|
&\leq& \sum_{\beta\leq\alpha} {\alpha\choose\beta}\frac{(\alpha-\beta)!}{r^{|\alpha-\beta|}}e^{2\varepsilon\sqrt{1+|x|^2}}\cdot
\sqrt{2}e^{2s}\frac{\beta !}{r^{|\beta|}}a(x,\xi)\\
&\leq& \sqrt{2}e^{2s}\frac{\alpha !}{r^{|\alpha|}}2^{|\alpha|}e^{2\varepsilon\sqrt{1+|x|^2}}a(x,\xi).
\eeqs
Using the property $ii)$ of the function $a(x,\xi)$, we get
\beq\label{15}
\left|D^{\alpha}_x \left(e^{\varepsilon\sqrt{1+|x|^2}}a(x,\xi)\right)\right|\leq \sqrt{2}e^{2s}
\frac{\alpha ! 2^{|\alpha|}}{r^{|\alpha|}}e^{2\varepsilon\sqrt{1+|x|^2}}a(x,\xi)\leq
\sqrt{2}e^{2s+2\varepsilon}\frac{\alpha ! 2^{|\alpha|}}{r^{|\alpha|}},\, \forall \xi\in K.
\eeq
By this estimate and proposition 7 of \cite{PBD} one has $e^{\varepsilon\sqrt{1+|x|^2}}a(x,\xi)$ is a multiplier for $\SSS'^{*}$. Because of (\ref{7}), (\ref{5}) is a subset of $\SSS'^{*}$. Now to prove that  (\ref{5}) is bounded in $\SSS'^{*}$. We will give the prove only in the $\{M_p\}$ case, the $(M_p)$ case is similar. Let $\psi\in\SSS^{\{M_p\}}$. There exists $h>0$ such that $\psi\in\SSS^{M_p,h}_{\infty}$. Note that
\beqs
\left\langle e^{\varepsilon\sqrt{1+|x|^2}}a(x,\xi)T(x)e^{-x\xi^{(k)}},\psi(x)\right\rangle=\left\langle T(x)e^{-x\xi^{(k)}},e^{\varepsilon\sqrt{1+|x|^2}}a(x,\xi)\psi(x)\right\rangle,\, \forall k\in\{1,...,l\}, \forall \xi\in K.
\eeqs
Choose $m\leq h/4$. By (\ref{15}), we have\\
$\ds\frac{m^{|\alpha|+|\beta|}\langle x\rangle^{\beta}\left|D^{\alpha}\left(e^{\varepsilon\sqrt{1+|x|^2}}a(x,\xi)\psi(x)\right)\right|}{M_{\alpha}M_{\beta}}$
\beqs
&\leq&m^{|\alpha|+|\beta|}\langle x\rangle^{\beta}\sum_{\gamma\leq\alpha}{\alpha\choose\gamma}\frac{\sqrt{2}e^{2s+2\varepsilon}(\alpha-\gamma) ! 2^{|\alpha-\gamma|}|D^{\gamma}\psi(x)|}{r^{|\alpha-\gamma|}M_{\alpha}M_{\beta}}\\
&\leq&C_1\sigma_{h,\infty}(\psi)\sum_{\gamma\leq\alpha}{\alpha\choose\gamma}\frac{h^{|\alpha|+|\beta|}(\alpha-\gamma) ! 2^{|\alpha-\gamma|}}{4^{|\alpha|+|\beta|}r^{|\alpha-\gamma|}M_{\alpha-\gamma}h^{|\gamma|+|\beta|}}\leq
C_1\sigma_{h,\infty}(\psi)\sum_{\gamma\leq\alpha}{\alpha\choose\gamma}\frac{h^{|\alpha|-|\gamma|}(\alpha-\gamma) ! } {2^{|\alpha|}r^{|\alpha-\gamma|}M_{\alpha-\gamma}}\\
&\leq& C\sigma_{h,\infty}(\psi),\, \forall \xi\in K.
\eeqs
Hence $e^{\varepsilon\sqrt{1+|x|^2}}a(x,\xi)T(x)e^{-x\xi^{(k)}}$, $\xi\in K$, is bounded in $\SSS'^{\{M_p\}}$. Buy (\ref{7}), the set (\ref{5}) is bounded in $\SSS'^{\{M_p\}}$.\\
\indent We will prove that $e^{-\varepsilon\sqrt{1+|x|^2}}\in \SSS^{*}$. In order to do that we will estimate the derivatives of $e^{-\varepsilon\sqrt{1+|x|^2}}$ with the Cauchy integral formula (similarly as for $e^{\varepsilon\sqrt{1+|x|^2}}$). We obtain
\beqs
\left|\partial_z^{\alpha} e^{-\varepsilon\sqrt{1+|x|^2}}\right|\leq\frac{\alpha !}{r^{|\alpha|}}
\sup_{|w_1-x_1|\leq r,...,|w_d-x_d|\leq r}\left|e^{-\varepsilon\sqrt{1+w^2}}\right|,
\eeqs
where, $0<r<1/(8d)$. Let $w=u+iv$. Then, if we put $\ds\rho=\sqrt{\left(1+|u|^2-|v|^2\right)^2+4(uv)^2}$, $\ds\cos\theta= \frac{1+|u|^2-|v|^2}{\sqrt{\left(1+|u|^2-|v|^2\right)^2+4(uv)^2}}$, $\ds\sin\theta= \frac{2uv}
{\sqrt{\left(1+|u|^2-|v|^2\right)^2+4(uv)^2}}$ (where $\theta\in(-\pi,\pi)$), we have that $\theta\in(-\pi/2,\pi/2)$ (because $\cos\theta>0$ and $\theta\in(-\pi,\pi)$) and
\beqs
\mathrm{Re\,}\sqrt{1+|u|^2-|v|^2+2iuv}&=&\mathrm{Re\,}\sqrt{\rho(\cos\theta+i\sin\theta)}=
\mathrm{Re\,}\sqrt{\rho}\left(\cos\frac{\theta}{2}+i\sin\frac{\theta}{2}\right) =\sqrt{\rho}\cos\frac{\theta}{2}\geq\frac{\sqrt{\rho}}{2},
\eeqs
where the second equality holds because we take the principal branch of $\sqrt{z}$. Because $r<1/(8d)$, we get
\beqs
\left|e^{-\varepsilon\sqrt{1+w^2}}\right|&=& e^{\mathrm{Re\,}\left(-\varepsilon\sqrt{1+w^2}\right)}\leq
e^{-\frac{\varepsilon}{2}\sqrt[4]{\left(1+|u|^2-|v|^2\right)^2+4(uv)^2}}\leq
e^{-\frac{\varepsilon}{2}\sqrt{1+|u|^2-|v|^2}}\\
&\leq& e^{-\frac{\varepsilon}{2}\sqrt{1+\frac{|x|^2}{2}-|u-x|^2-|v|^2}}\leq
e^{-\frac{\varepsilon}{4}\sqrt{1+|x|^2}}.
\eeqs
Hence, we obtain
\beq\label{17}
\left|\partial_x^{\alpha} e^{-\varepsilon\sqrt{1+|x|^2}}\right|\leq\frac{\alpha !}{r^{|\alpha|}}
e^{-\frac{\varepsilon}{4}\sqrt{1+|x|^2}}.
\eeq
From this, it easily follows that $e^{-\varepsilon\sqrt{1+|x|^2}}\in\SSS^{*}$. So $e^{-x\xi}T(x)\in\SSS'^*\left(\RR^d_x\right)$, for $\xi\in K$, because $e^{-x\xi}T(x)=T(x)e^{-x\xi+\varepsilon\sqrt{1+|x|^2}}e^{-\varepsilon\sqrt{1+|x|^2}}$ and we proved that $T(x)e^{-x\xi+\varepsilon\sqrt{1+|x|^2}}\in\SSS'^*\left(\RR^d_x\right)$, for $\xi\in K$.\\
\indent Put $f(\xi+i\eta)=\mathcal{F}_{x\rightarrow\eta}(e^{-x\xi}T(x))$. We will prove that $f$ is an analytic function on $\mathrm{ch\,}B+i\RR^d$. Let $U$ be an arbitrary bounded open subset of $\mathrm{ch\,}B$ such that $K=\overline{U}\subset\subset \mathrm{ch\,}B$. For $\psi\in\SSS^{*}$ and $\xi\in U$, we have
\beqs
\langle f(\xi+i\eta),\psi(\eta)\rangle&=&\left\langle \mathcal{F}_{x\rightarrow\eta}\left(e^{-x\xi}T(x)\right),\psi(\eta)\right\rangle=\left\langle e^{-x\xi}T(x),\mathcal{F}(\psi)(x)\right\rangle\\
&=&\left\langle e^{-x\xi}T(x),\int_{\RR^d}e^{-ix\eta}\psi(\eta)d\eta\right\rangle
=\left\langle e^{\varepsilon\sqrt{1+|x|^2}}e^{-x\xi}T(x),
e^{-\varepsilon\sqrt{1+|x|^2}}\int_{\RR^d}e^{-ix\eta}\psi(\eta)d\eta\right\rangle\\
&=&\left\langle \left(e^{\varepsilon\sqrt{1+|x|^2}}e^{-x\xi}T(x)\right)\otimes 1_{\eta},
e^{-\varepsilon\sqrt{1+|x|^2}}e^{-ix\eta}\psi(\eta)\right\rangle\\
&=&\int_{\RR^d}\left\langle e^{\varepsilon\sqrt{1+|x|^2}}e^{-x\xi}T(x)e^{-ix\eta},e^{-\varepsilon\sqrt{1+|x|^2}}\right\rangle\psi(\eta)d\eta.
\eeqs
Hence
\beq\label{20}
f(\xi+i\eta)=\left\langle e^{\varepsilon\sqrt{1+|x|^2}}e^{-x\xi}T(x)e^{-ix\eta},e^{-\varepsilon\sqrt{1+|x|^2}}\right\rangle.
\eeq
First we will prove that $f\in \mathcal{C}^{\infty}\left(U\times\RR^d_{\eta}\right)$. We will prove the differentiability only in $\xi_1$ and in the $\{M_p\}$ case. The existence of the rest of the derivatives is proved in analogous way and the $(M_p)$ case is treated similarly. Let $\xi^{(0)}=\left(\xi^{(0)}_1,...,\xi^{(0)}_d\right)=\left(\xi^{(0)}_1,\xi'\right)\in U$, $\xi=\left(\xi^{(0)}_1+\xi_1,\xi^{(0)}_2,...,\xi^{(0)}_d\right)=\left(\xi^{(0)}_1+\xi_1,\xi'\right)$, $x=(x_1,...,x_d)=(x_1,x')$. Let $0<|\xi_1|<\delta<\varepsilon<1$ such that the ball with radius $\delta$ and center in $\xi^{(0)}$ is contained in $U$. Then, by using (\ref{7}) and (\ref{20}), we obtain\\
$\ds\frac{f(\xi+i\eta)-f(\xi^{(0)}+i\eta)}{\xi_1}-\left\langle e^{\varepsilon\sqrt{1+|x|^2}}(-x_1)e^{-x\xi^{(0)}}T(x)e^{-ix\eta},e^{-\varepsilon\sqrt{1+|x|^2}}\right\rangle$
\beqs
=\sum_{k=1}^l\left\langle e^{-ix\eta} e^{-x\xi^{(k)}}T(x)e^{\varepsilon\sqrt{1+|x|^2}}\left(\frac{a(x,\xi)-a\left(x,\xi^{(0)}\right)}{\xi_1}+x_1 a\left(x,\xi^{(0)}\right)\right),e^{-\varepsilon\sqrt{1+|x|^2}}\right\rangle.
\eeqs
It is enough to prove that, for every $\psi\in\SSS^{\{M_p\}}$,
\beqs
\ds e^{\varepsilon\sqrt{1+|x|^2}}\left(\frac{a(x,\xi)-a\left(x,\xi^{(0)}\right)}{\xi_1}+x_1 a\left(x,\xi^{(0)}\right)\right)\psi(x)\longrightarrow 0, \mbox{ when } \xi_1\longrightarrow 0, \mbox{ in } \SSS^{\{M_p\}}.
\eeqs
First note that
\beqs
e^{\varepsilon\sqrt{1+|x|^2}}\left(\frac{a(x,\xi)-a\left(x,\xi^{(0)}\right)}{\xi_1}+x_1 a\left(x,\xi^{(0)}\right)\right)=
e^{\varepsilon\sqrt{1+|x|^2}}a\left(x,\xi^{(0)}\right)\left(\frac{e^{-x_1\xi_1}-1}{\xi_1}+x_1\right).
\eeqs
Now, we get
\beqs
\frac{e^{-x_1\xi_1}-1}{\xi_1}+x_1=\frac{1}{\xi_1}\sum_{n=1}^{\infty}\frac{(-1)^nx_1^n\xi_1^n}{n!}+x_1=
\sum_{n=2}^{\infty}\frac{(-1)^nx_1^n\xi_1^{n-1}}{n!}.
\eeqs
So, for $j\in\NN$, $j\geq2$ and $0<|\xi_1|<\delta<\varepsilon<1$, we have
\beqs
\left|D^j_{x_1}\left(\frac{e^{-x_1\xi_1}-1}{\xi_1}+x_1\right)\right|&=&
\left|D^j_{x_1}\left(\sum_{n=2}^{\infty}\frac{(-1)^nx_1^n\xi_1^{n-1}}{n!}\right)\right|=
\left|\sum_{n=j}^{\infty}\frac{(-1)^n n!x_1^{n-j}\xi_1^{n-1}}{(n-j)!n!}\right|\\
&\leq& |\xi_1|\sum_{n=j}^{\infty}\frac{|x_1|^{n-j}|\xi_1|^{n-2}}{(n-j)!}\leq
|\xi_1|\sum_{n=j}^{\infty}\frac{|x_1|^{n-j}|\xi_1|^{n-j}}{(n-j)!}\leq\delta e^{|x_1|\delta}.
\eeqs
Using similar technic, we obtain the estimates
\beqs
\left|D_{x_1}\left(\frac{e^{-x_1\xi_1}-1}{\xi_1}+x_1\right)\right|\leq\delta |x_1| e^{|x_1|\delta} \mbox{ and } \left|\left(\frac{e^{-x_1\xi_1}-1}{\xi_1}+x_1\right)\right|\leq \delta|x_1|^2 e^{|x_1|\delta}.
\eeqs
So, in all cases, we have $\ds\left|D^j_{x_1}\left(\frac{e^{-x_1\xi_1}-1}{\xi_1}+x_1\right)\right|\leq \delta\langle x_1\rangle^2 e^{|x_1|\delta}.$ By using (\ref{15}), we get (for simpler notation we write $j$ for the $d$-tuple $(j,0,...,0)$)\\
$\ds\left|D^{\alpha}\left(e^{\varepsilon\sqrt{1+|x|^2}}a\left(x,\xi^{(0)}\right)
\left(\frac{e^{-x_1\xi_1}-1}{\xi_1}+x_1\right)\psi(x)\right)\right|$
\beqs
&=&\left|\sum_{\beta\leq\alpha}\sum_{j\leq\beta}{\alpha\choose\beta}{\beta\choose j} D^{\beta-j}\left(e^{\varepsilon\sqrt{1+|x|^2}}a\left(x,\xi^{(0)}\right)\right)
D^j\left(\frac{e^{-x_1\xi_1}-1}{\xi_1}+x_1\right)D^{\alpha-\beta}\psi(x)\right|\\
&\leq&\sum_{\beta\leq\alpha}\sum_{j\leq\beta}{\alpha\choose\beta}{\beta\choose j} \sqrt{2}e^{2s}
\frac{(\beta-j) ! 2^{|\beta-j|}}{r^{|\beta-j|}}e^{2\varepsilon\sqrt{1+|x|^2}}a\left(x,\xi^{(0)}\right)
\delta\langle x_1\rangle^2 e^{|x_1|\delta}|D^{\alpha-\beta}\psi(x)|\\
&\leq&C\delta\langle x_1\rangle^2\sum_{\beta\leq\alpha}\sum_{j\leq\beta}{\alpha\choose\beta}{\beta\choose j}
\left(\frac{2}{r}\right)^{|\beta-j|}(\beta-j) !|D^{\alpha-\beta}\psi(x)|,
\eeqs
where we used the inequality $e^{2\varepsilon\sqrt{1+|x|^2}}a(x,\xi^{(0)})e^{|x_1|\delta}
\leq e^{3\varepsilon\sqrt{1+|x|^2}}a(x,\xi^{(0)})\leq e^{3\varepsilon}$, which follows from the property $ii)$ of $a(x,\xi)$. Because $\psi\in\SSS^{\{M_p\}}$, there exists $m>0$ such that $\psi\in\SSS^{M_p,m}_{\infty}$. Choose $h$ such that $h<m/4$, $h<1/4$ and $hH<m$. We get\\
$\ds\frac{\ds h^{|\alpha|+|\beta|}\langle x\rangle^{\beta}\left|D^{\alpha}\left(e^{\varepsilon\sqrt{1+|x|^2}}a\left(x,\xi^{(0)}\right)
\left(\frac{e^{-x_1\xi_1}-1}{\xi_1}+x_1\right)\psi(x)\right)\right|}{M_{\alpha}M_{\beta}}$
\beqs
&\leq& C\delta\sum_{\gamma\leq\alpha}\sum_{j\leq\gamma}{\alpha\choose\gamma}{\gamma\choose j}\left(\frac{2}{r}\right)^{|\gamma-j|}(\gamma-j) !\frac{\langle x_1\rangle^2\langle x\rangle^{|\beta|}h^{|\alpha|+|\beta|}|D^{\alpha-\gamma}\psi(x)|}{M_{\alpha-\gamma}M_{\gamma-j}M_jM_{\beta}}\\
&\leq&C_1\delta\sum_{\gamma\leq\alpha}\sum_{j\leq\gamma}{\alpha\choose\gamma}{\gamma\choose j}\left(\frac{2}{r}\right)^{|\gamma-j|}(\gamma-j) !\frac{\langle x\rangle^{|\beta|+2}h^{|\alpha|+|\beta|} H^{|\beta|+2}|D^{\alpha-\gamma}\psi(x)|}{M_{\alpha-\gamma}M_{\gamma-j}M_jM_{\beta+2}}\\
&\leq&C_2\delta\sigma_{m,\infty}(\psi)\sum_{\gamma\leq\alpha}\sum_{j\leq\gamma}{\alpha\choose\gamma}{\gamma\choose j}\left(\frac{2}{r}\right)^{|\gamma-j|}(\gamma-j) !\frac{h^{|\alpha|+|\beta|} H^{|\beta|}}{m^{|\alpha|-|\gamma|}m^{|\beta|+2}M_{\gamma-j}M_j}\\
&\leq&C_3\delta\sigma_{m,\infty}(\psi)\sum_{\gamma\leq\alpha}\sum_{j\leq\gamma}{\alpha\choose\gamma}{\gamma\choose j}\left(\frac{2}{r}\right)^{|\gamma-j|}\left(\frac{h}{m}\right)^{|\alpha|-|\gamma|}
\left(\frac{hH}{m}\right)^{|\beta|}\frac{h^{|\gamma|}(\gamma-j) !}{M_{\gamma-j}M_j}
\leq C_0\delta\sigma_{m,\infty}(\psi),
\eeqs
where we use $(M.2)$ and the fact $\ds\frac{k^p p!}{M_p}\rightarrow 0$, when $p\rightarrow\infty$. Now, from this it follows that
\beqs
e^{\varepsilon\sqrt{1+|x|^2}}\left(\frac{a(x,\xi)-a\left(x,\xi^{(0)}\right)}{\xi_1}+x_1 a\left(x,\xi^{(0)}\right)\right)\psi(x)\longrightarrow 0,\,
\xi_1\longrightarrow 0
\eeqs
in $\SSS^{\{M_p\}}$ and by the above remarks, the differentiability of $f(\xi+i\eta)$ on $U\times \RR^d_{\eta}$ follows. Also, from the previous, we can conclude that $\partial_{\xi}^{\alpha}f(\xi+i\eta)=\left\langle e^{\varepsilon\sqrt{1+|x|^2}}(-x)^{\alpha}e^{-x\xi}T(x)e^{-ix\eta},e^{-\varepsilon\sqrt{1+|x|^2}}\right\rangle$ and similarly $\partial_{\eta}^{\alpha}f(\xi+i\eta)=\left\langle e^{\varepsilon\sqrt{1+|x|^2}}(-ix)^{\alpha}e^{-x\xi}T(x)e^{-ix\eta},e^{-\varepsilon\sqrt{1+|x|^2}}\right\rangle$. From this and the arbitrariness of $U$, the analyticity of $f(\xi+i\eta)$ follows because it satisfies the Cauchy-Riemann equations. So, for $\zeta=\xi+i\eta$, we get
\beq\label{25}
f(\zeta)=\left\langle e^{\varepsilon\sqrt{1+|x|^2}}e^{-x\zeta}T(x),e^{-\varepsilon\sqrt{1+|x|^2}}\right\rangle
\eeq
and $\partial_{\zeta}^{\alpha}f(\zeta)=\left\langle e^{\varepsilon\sqrt{1+|x|^2}}(-x)^{\alpha}e^{-x\zeta}T(x),e^{-\varepsilon\sqrt{1+|x|^2}}\right\rangle$, for $\zeta\in U+i\RR^d_{\eta}$, for each fixed $U$ ($\varepsilon$ depends on $U$).\\
\indent Now we will prove the estimates (\ref{3}) for $f(\xi+i\eta)$. Let $K\subset\subset \mathrm{ch\,}B$ be arbitrary but fixed. First we will consider the $(M_p)$ case. We know that $\SSS^{(M_p)}$ is a $(FS)$ - space and $\ds\SSS^{(M_p)}=\lim_{\substack{\longleftarrow\\ h\rightarrow\infty}}\SSS^{M_p,h}_{\infty}$. If we denote the closure of $\SSS^{(M_p)}$ in $\SSS^{M_p,h}_{\infty}$ by $\widetilde{\SSS}^{M_p,h}_{\infty}$ then $\ds\SSS^{(M_p)}=\lim_{\substack{\longleftarrow\\ h\rightarrow\infty}}\widetilde{\SSS}^{M_p,h}_{\infty}$ and the projective limit is reduced. Then $\ds\SSS'^{(M_p)}=\lim_{\substack{\longrightarrow\\ h\rightarrow\infty}}\widetilde{\SSS}'^{M_p,h}_{\infty}$ which is injective inductive limit with compact maps (because the projective limit is with compact maps). Because we proved that the set $\left\{S\in\DD'^{*}|S(x)=T(x)e^{-x\xi+\varepsilon\sqrt{1+|x|^2}},\xi\in K\right\}$ is bounded in $\SSS'^{(M_p)}$, it follows that there exists $h>0$ such that $\left\{S\in\DD'^{*}|S(x)=T(x)e^{-x\xi+\varepsilon\sqrt{1+|x|^2}},\xi\in K\right\}\subseteq \widetilde{\SSS}'^{M_p,h}_{\infty}$ and it's bounded there. By (\ref{17}), we have the estimate
\beqs
\frac{h^{|\alpha|+|\beta|}\langle x\rangle^{\beta}\left|D^{\alpha}_x \left(e^{-ix\eta}e^{-\varepsilon\sqrt{1+|x|^2}}\right)\right|}{M_{\alpha}M_{\beta}} &\leq&\sum_{\gamma\leq\alpha}{\alpha\choose\gamma}
\frac{(2h)^{|\alpha|-|\gamma|}(2h)^{|\gamma|}h^{|\beta|}\langle x\rangle^{\beta}
|\eta|^{\gamma}(\alpha-\gamma) !e^{-\frac{\varepsilon}{4}\sqrt{1+|x|^2}}}
{2^{|\alpha|}r^{|\alpha-\gamma|}M_{\alpha-\gamma}M_{\gamma}M_{\beta}}\\
&\leq& C_1\frac{1}{2^{|\alpha|}}\sum_{\gamma\leq\alpha}{\alpha\choose\gamma}\left(\frac{2h}{r}\right)^{|\alpha|-|\gamma|}
\frac{(\alpha-\gamma) !e^{M(h\langle x\rangle)}e^{M(2h|\eta|)}e^{-\frac{\varepsilon}{4}\langle x\rangle}}
{M_{\alpha-\gamma}}\\
&\leq& C' e^{M(2h|\eta|)},
\eeqs
where we use that $e^{M(h\langle x\rangle)}e^{-\frac{\varepsilon}{4}\langle x\rangle}$ is bounded and $\ds\frac{k^p p!}{M_p}\rightarrow 0$ when $p\rightarrow\infty$. Then, for $\xi\in K$ and $\eta\in\RR^d$,
\beqs
|f(\xi+i\eta)|=\left|\left\langle e^{\varepsilon\sqrt{1+|x|^2}}e^{-x\xi}T(x),e^{-ix\eta}e^{-\varepsilon\sqrt{1+|x|^2}}\right\rangle\right|\leq
C\left\|e^{-ix\eta}e^{-\varepsilon\sqrt{1+|x|^2}}\right\|_{\widetilde{\SSS}^{M_p,h}_{\infty}}\leq \tilde{C}e^{M(2h|\eta|)}.
\eeqs
Now we will consider the $\{M_p\}$ case. $\SSS^{\{M_p\}}$ is a $(DFS)$ - space and $\ds\SSS^{\{M_p\}}=\lim_{\substack{\longrightarrow\\ h\rightarrow 0}}\SSS^{M_p,h}_{\infty}$, where the inductive limit is injective with compact maps. Let $h>0$ be fixed. For shorter notation, denote by $F$ the set $\left\{S\in\DD'^{*}|S(x)=T(x)e^{-x\xi+\varepsilon\sqrt{1+|x|^2}},\xi\in K\right\}$ and by $J$ the inclusion $\SSS^{M_p,h}_{\infty}\longrightarrow \SSS^{\{M_p\}}$. Because we already proved that $F$ is a bounded subset of $\SSS'^{\{M_p\}}$, its image under ${}^{t}J$ (the transposed mapping of $J$) is a bounded subset of $\SSS'^{M_p,h}_{\infty}$. By the above calculations we see that $e^{-ix\eta}e^{-\varepsilon\sqrt{1+|x|^2}}$ is in $\SSS^{M_p,m}_{\infty}$, for every $m>0$. Hence, for $\xi\in K$ and $\eta\in\RR^d$, we have
\beqs
|f(\xi+i\eta)|&=&\left|\left\langle e^{\varepsilon\sqrt{1+|x|^2}}e^{-x\xi}T(x),e^{-ix\eta}e^{-\varepsilon\sqrt{1+|x|^2}}\right\rangle\right|
= \left|\left\langle {}^{t}J\left(e^{\varepsilon\sqrt{1+|x|^2}}e^{-x\xi}T(x)\right), e^{-ix\eta}e^{-\varepsilon\sqrt{1+|x|^2}}\right\rangle\right|\\
&\leq& C'_h\left\|e^{-ix\eta}e^{-\varepsilon\sqrt{1+|x|^2}}\right\|_{\SSS^{M_p,h}_{\infty}}\leq C_h e^{M(2h|\eta|)},
\eeqs
where we used the above estimate for $\ds\frac{h^{|\alpha|+|\beta|}\langle x\rangle^{\beta}\left|D^{\alpha}\left(e^{-ix\eta}e^{-\varepsilon\sqrt{1+|x|^2}}\right)\right|}
{M_{\alpha}M_{\beta}}$.
\end{proof}

\begin{remark}
If, for $S\in\DD'^{*}$, the conditions of the theorem are fulfilled, we call $\mathcal{F}_{x\rightarrow\eta}\left(e^{-x\xi}S(x)\right)$ the Laplace transform of $S$ and denote it by $\mathcal{L}(S)$. Moreover, by (\ref{25}),
\beqs
\mathcal{L}(S)(\zeta)=\left\langle e^{\varepsilon\sqrt{1+|x|^2}}e^{-x\zeta}S(x),e^{-\varepsilon\sqrt{1+|x|^2}}\right\rangle, \mbox{ for } \zeta\in U+i\RR^d_{\eta},
\eeqs
where $\overline{U}\subset\subset \mathrm{ch\,}B$ and $\varepsilon$ depends on $U$.\\
\indent Note that, if for $S\in\DD'^{*}$ the conditions of the theorem are fulfilled for $B=\RR^d$, then the choice of $\varepsilon$ can be made uniform for all $K\subset\subset\RR^d$.
\end{remark}

For the next theorem we need the following technical results.
\begin{lemma}\label{psss}
Let $(k_p)\in\mathfrak{R}$. There exists $(k'_p)\in\mathfrak{R}$ such that $k'_p\leq k_p$ and $\ds\prod_{j=1}^{p+q}k'_j\leq 2^{p+q}\prod_{j=1}^{p}k'_j\cdot\prod_{j=1}^{q}k'_j$, for all $p,q\in\ZZ_+$.
\end{lemma}
\begin{proof} Define $k'_1=k_1$ and inductively $\ds k'_j=\min\left\{k_j,\frac{j}{j-1}k'_{j-1}\right\}$, for $j\geq 2$, $j\in\NN$. Obviously $k'_j\leq k_j$ and one easily checks that $(k'_j)$ is monotonically increasing. To prove that $k'_j$ tends to infinity, suppose the contrary. Then, because $(k'_j)$ is a monotonically increasing sequence of positive numbers, it follows that it is bounded by some $C>0$. Because $(k_j)\in\mathfrak{R}$, there exists $j_0$, such that, for all $j\geq j_0$, $j\in\NN$, $k_j\geq 2C$. So, for all $j\geq j_0+1$, $\ds k'_j=\frac{j}{j-1}k'_{j-1}$. We get that $\ds k'_j=\frac{j}{j_0}k'_{j_0}\rightarrow \infty$, when $j\longrightarrow \infty$, which is a contradiction. Hence $(k'_j)\in\mathfrak{R}$. Note that, for all $p,j\in\ZZ_+$, we have $\ds k'_{p+j}\leq \frac{p+j}{j}k'_{j}$. Hence $\ds \prod_{j=1}^{p+q}k'_j=\prod_{j=1}^{p}k'_j\cdot\prod_{j=1}^{q}k'_{p+j}\leq \prod_{j=1}^{p}k'_j\cdot\prod_{j=1}^{q}\frac{p+j}{j}k'_{j}=\frac{(p+q)!}{p!q!}\prod_{j=1}^{p}k'_j\cdot
\prod_{j=1}^{q}k'_{j}\leq 2^{p+q}\prod_{j=1}^{p}k'_j\cdot\prod_{j=1}^{q}k'_{j}$.
\end{proof}

\indent We will construct certain class of ultrapolynomials similar to those in \cite{Komatsu1}, (see (10.9)' in \cite{Komatsu1}), which will have the added beneficence of not having zeroes in a strip containing the real axis.\\
\indent Let $c>0$ be fixed. Let $k>0$, $l>0$ and $(k_p)\in\mathfrak{R}$, $(l_p)\in\mathfrak{R}$ be arbitrary but fixed. Choose $q\in\ZZ_+$ such that $\ds \frac{c\sqrt{d}}{l m_p}<\frac{1}{2}$, for all $p\in\NN$, $p\geq q$ in the $(M_p)$ case and $\ds \frac{c\sqrt{d}}{l_p m_p}<\frac{1}{2}$, for all $p\in\NN$, $p\geq q$ in the $\{M_p\}$ case. Consider the entire functions
\beq\label{u1}
P_l(w)=\prod_{j=q}^{\infty}\left(1+\frac{w^2}{l^2 m_j^2}\right),\,\, w\in\CC^d
\eeq
in the $(M_p)$ case, resp.
\beq\label{u2}
P_{l_p}(w)=\prod_{j=q}^{\infty}\left(1+\frac{w^2}{l_j^2 m_j^2}\right),\,\, w\in\CC^d
\eeq
in the $\{M_p\}$ case. It is easily checked that the entire function $P_l(w_1,0,...,0)$, resp. $P_{l_p}(w_1,0,...,0)$, of one variable satisfies the condition c) of proposition 4.6 of \cite{Komatsu1}. Hence, $P_l(w)$, resp. $P_{l_p}(w)$, satisfies the equivalent conditions a) and b) of proposition 4.5 of \cite{Komatsu1}. Hence, there exist $L>0$ and $C'>0$, resp. for every $L>0$ there exists $C'>0$, such that $|P_l(w)|\leq C'e^{M(L|w|)}$, resp. $|P_{l_p}(w)|\leq C'e^{M(L|w|)}$, for all $w\in\CC^d$ and $P_l(D)$, resp. $P_{l_p}(D)$, are ultradifferential operators of $(M_p)$, resp. $\{M_p\}$, type. It is easy to check that $P_l(w)$ and $P_{l_p}(w)$ don't have zeroes in $W=\RR^d+i\{v\in\RR^d||v_j|\leq c,\,j=1,...,d\}$. For $w=u+iv\in W$, $|u|\geq 2c\sqrt{d}$, we have $\ds \left|w^2\right|\geq \frac{|w|^2}{4}$ and $\ds \left|1+\frac{w^2}{l_j^2 m_j^2}\right|\geq 1$, for $j\geq q$. We estimate as follows
\beqs
|P_{l_p}(w)|&=&\left|\prod_{j=q}^{\infty}\left(1+\frac{w^2}{l_j^2 m_j^2}\right)\right|=\sup_p\prod_{j=q}^{p}\left|1+\frac{w^2}{l_j^2 m_j^2}\right|\geq\sup_p\prod_{j=q}^{p}\frac{\left|w^2\right|}{l_j^2 m_j^2}\geq\sup_p\prod_{j=q}^{p}\frac{|w|^2}{4l_j^2 m_j^2}\\
&=&\frac{\prod_{j=1}^{q-1} 4l_j^2}{|w|^{2q-2}}\left(\sup_p\frac{|w|^pM_{q-1}}{M_p\prod_{j=1}^p 2l_j}\right)^2
=C'_0\left(\frac{M_{q-1}\prod_{j=1}^{q-1} k_j}{|w|^{q-1}}\right)^2 e^{2N_{2l_p}(|w|)}\geq C'_0\frac{e^{N_{2l_p}(|w|)}}{e^{2N_{k_p}(|w|)}},
\eeqs
where we put $\ds C'_0=\prod_{j=1}^{q-1}\frac{4l_j^2}{k_j^2}$ and $l_p=l$ and $k_p=k$ in the $(M_p)$ case. For $w\in W$, because $P_l(w)$, resp. $P_{l_p}(w)$, doesn't have zeroes in $W$, we get that there exist $C_0>0$ such that
\beq\label{uu}
|P_l(w)|\geq C_0e^{-2M(|w|/k)}e^{M\left(|w|/(2l)\right)},\, \mbox{resp.}\, |P_{l_p}(w)|\geq C_0e^{-2N_{k_p}(|w|)}e^{N_{2l_p}(|w|)},\, w\in W.
\eeq
Now, by using Cauchy integral formula, we can estimate the derivatives of $1/P_l(x)$, resp. $1/P_{l_p}(\xi)$. We will introduce some notations to make the calculations less cumbersome. For $r>0$, denote by $B_r(a)$ the polydisc with center at $a$ and radii $r$, i.e. $\{z\in\CC^d||z_j-a_j|< r,\, j=1,2,...,d\}$ and by $T_r(a)$ the corresponding polytorus $\{z\in\CC^d||z_j-a_j|= r,\, j=1,2,...,d\}$. We will do it for the $\{M_p\}$ case, for the $(M_p)$ case it is similar. We already know that on $W$, $1/P_{l_p}(w)$ is analytic function ($P_{l_p}$ doesn't have zeroes in $W$). Hence
\beqs
\left|\partial^{\alpha}_w\frac{1}{P_{l_p}(x)}\right|\leq \frac{\alpha!}{r^{|\alpha|}}\cdot\left\|\frac{1}{P_{l_p}(z)}\right\|_{L^{\infty}(T_r(x))}
\leq \frac{\alpha!}{C_0r^{|\alpha|}}\cdot
\left\|\frac{e^{2N_{k_p}(|z|)}}{e^{N_{2l_p}(|z|)}}\right\|_{L^{\infty}(T_r(x))},
\eeqs
for arbitrary but fixed $r\leq c$ (so $\overline{B_{r}(x)}\subseteq W$). For $x\in\RR^d\backslash B_{2r\sqrt{d}}(0)$, there exists $j\in\{1,...,d\}$ such that $|x_j|\geq 2r\sqrt{d}$. Then, on $T_r(x)$, $|z|\geq |x|-|z-x|=|x|-r\sqrt{d}\geq |x|/2$, i.e. $e^{N_{2l_p}(|z|)}\geq e^{N_{2l_p}(|x|/2)}=e^{N_{4l_p}(|x|)}$. Moreover, for such $x$, we have
\beqs
e^{2N_{k_p}(|z|)}\leq e^{2N_{k_p}(|x|+r\sqrt{d})}\leq 4e^{2N_{k_p}(2r\sqrt{d})}e^{2N_{k_p}(2|x|)}=C_1e^{2N_{k_p}(2|x|)},
\eeqs
where in the last inequality we used that $e^{M(\lambda+\nu)}\leq 2e^{M(2\lambda)}e^{M(2\nu)}$, for $\lambda\geq0$, $\nu\geq 0$. So, we obtain $\ds \left|\partial^{\alpha}_w\frac{1}{P_{l_p}(x)}\right|\leq C\cdot\frac{\alpha!}{r^{|\alpha|}}\frac{e^{2N_{k_p}(2|x|)}}{e^{N_{4l_p}(|x|)}}$. For $x$ in $B_{2r\sqrt{d}}(0)$, $\ds\left\|e^{2N_{k_p}(|z|)}e^{-N_{2l_p}(|z|)}\right\|_{L^{\infty}(T_r(x))}$ is bounded, so we can conclude that the above inequality holds, possible with another constant $C$. Analogously, we can prove that, for the $(M_p)$ case, $\ds \left|\partial^{\alpha}_w\frac{1}{P_l(x)}\right|\leq C\cdot\frac{\alpha!}{r^{|\alpha|}}\frac{e^{2M\left(2|x|/k\right)}}{e^{M\left(|x|/(4l)\right)}}$. This is important, because, if $k>0$ is fixed, resp. $(k_p)\in\mathfrak{R}$ is fixed, then we can find $l>0$, resp. $(l_p)\in\mathfrak{R}$, such that $\ds e^{2M\left(2|x|/k\right)}e^{-M\left(|x|/(4l)\right)}\leq C''e^{-M\left(|x|/k\right)}$, resp. $e^{2N_{k_p}(2|x|)}e^{-N_{4l_p}(|x|)}\leq C''e^{-N_{k_p}(|x|)}$, for some $C''>0$. This inequality trivially follows from proposition 3.6 of \cite{Komatsu1} in the $(M_p)$ case. To prove the inequality in the $\{M_p\}$ case, first note that $e^{2N_{k_p}(2|x|)}e^{N_{k_p}(|x|)}\leq e^{3N_{k_p/2}(|x|)}$. By lemma \ref{psss}, there exists $(k'_p)\in \mathfrak{R}$ such that $k'_p\leq k_p/2$ and $\ds\prod_{j=1}^{p+q}k'_j\leq 2^{p+q}\prod_{j=1}^{p}k'_j\cdot\prod_{j=1}^{q}k'_j$, for all $p,q\in\ZZ_+$. So $\ds e^{3N_{k_p/2}(|x|)}\leq e^{3N_{k'_p}(|x|)}$. If we put $N_0=1$ and $\ds N_p=M_p\prod_{j=1}^{p}k'_j$, for $p\in\ZZ_+$, then, by the properties of $(k'_p)$, it follows that $N_p$ satisfies $(M.1)$, $(M.2)$ and $(M.3)'$ where the constant $H$ in $(M.2)$ for this sequence is equal to $2H$. Moreover, note that $N(\lambda)=N_{k'_p}(\lambda)$, for all $\lambda\geq 0$. We can now use proposition 3.6 of \cite{Komatsu1} for $N(|x|)$ (i.e. for $N_{k'_p}(|x|)$) and obtain $e^{3N_{k'_p}(|x|)}\leq c''e^{N_{k'_p}(4H^2|x|)}=c''e^{N_{k'_p/(4H^2)}(|x|)}$, for some $c''>0$. Now take $l_p$ such that $4l_p=k'_p/(4H^2)$, $p\in\ZZ_+$ and the desired inequality follows. So, we obtain
\beqs
\left|\partial^{\alpha}_x\frac{1}{P_l(x)}\right|\leq C\cdot\frac{\alpha!}{r^{|\alpha|}}e^{-M\left(|x|/k\right)},\, \mbox{resp.}\, \left|\partial^{\alpha}_x\frac{1}{P_{l_p}(x)}\right|\leq C\cdot\frac{\alpha!}{r^{|\alpha|}}e^{-N_{k_p}(|x|)},\, x\in\RR^d,\alpha\in\NN^d,
\eeqs
where $C$ depends on $k$ and $l$, resp. $(k_p)$ and $(l_p)$, and $M_p$; $r\leq c$ arbitrary but fixed. Moreover, from the above observation and (\ref{uu}), we obtain
\beq\label{uu1}
|P_l(w)|\geq \tilde{C}e^{M(|w|/k)},\, \mbox{resp.}\, |P_{l_p}(w)|\geq \tilde{C}e^{N_{k_p}(|w|)},\, w\in W,
\eeq
for some $\tilde{C}>0$.

\begin{lemma}\label{nnl}
let $g:[0,\infty)\longrightarrow[0,\infty)$ be an increasing function that satisfies the following estimate:\\
\indent for every $L>0$ there exists $C>0$ such that $g(\rho)\leq M(L\rho)+\ln C$.\\
Then there exists subordinate function $\epsilon(\rho)$ such that $g(\rho)\leq M(\epsilon(\rho))+\ln C'$, for some constant $C'>1$.
\end{lemma}
For the definition of subordinate function see \cite{Komatsu1}.
\begin{proof} If $g(\rho)$ is bounded then the claim of the lemma is trivial (we can take $C'$ large enough such that the inequality will hold for arbitrary subordinate function). Assume that $g$ is not bounded. We can easily find continuous strictly increasing function $f:[0,\infty)\longrightarrow[0,\infty)$ which majorizes $g$ such that for every $L>0$ there exists $C>0$ such that $f(\rho)\leq M(L\rho)+\ln C$. Hence, there exists $\rho_1>0$ such that $f(\rho)>0$ for $\rho\geq\rho_1$. There exists $\rho_0>0$ such that $M(\rho)=0$ for $\rho\leq\rho_0$ and $M(\rho)>0$ for $\rho>\rho_0$. Because $M(\rho)$ is continuous and strictly increasing on the interval $[\rho_0,\infty)$ and $\ds\lim_{\rho\rightarrow\infty}M(\rho)=\infty$, $M$ is bijection from $[\rho_0,\infty)$ to $[0,\infty)$ with continuous and strictly increasing inverse $M^{-1}:[0,\infty)\longrightarrow[\rho_0,\infty)$. Define $\epsilon(\rho)$ on $[\rho_1,\infty)$ in the following way $\epsilon(\rho)=M^{-1}(f(\rho))$ and define it linearly on $[0,\rho_1)$ such that it will be continuous on $[0,\infty)$ and $\epsilon(0)=0$. Then $\epsilon(\rho)$ is strictly increasing and continuous on $[0,\infty)$. Moreover, for $\rho\in[\rho_1,\infty)$, it satisfies $f(\rho)=M(\epsilon(\rho))$. Hence, there exists $C'>1$ such that $f(\rho)\leq M(\epsilon(\rho))+\ln C'$, for $\rho\geq 0$. It remains to prove that $\epsilon(\rho)/\rho\longrightarrow 0$ when $\rho\longrightarrow\infty$. Assume the contrary. Then, there exist $L>0$ and a strictly increasing sequence $\rho_j$ which tends to infinity when $j\longrightarrow\infty$, such that $\epsilon(\rho_j)\geq2L\rho_j$, i.e. $f(\rho_j)\geq M(2L\rho_j)$. For this $L$, by the condition for $f$, choose $C>1$ such that $f(\rho)\leq M(L\rho)+\ln C$. Then we have $M(2L\rho_j)\leq M(L\rho_j)+\ln C$, which contradicts the fact that $e^{M(\rho)}$ increases faster then $\rho^p$ for any $p$. One can obtain this contradiction by using equality (3.11) of \cite{Komatsu1}.
\end{proof}

\begin{theorem}\label{t2}
Let $B$ be a connected open set in $\RR^d_{\xi}$ and $f$ an analytic function on $B+i\RR^d_{\eta}$. Let $f$ satisfies the condition:\\
\indent for every compact subset $K$ of $B$ there exist $C>0$ and $k>0$, resp. for every $k>0$ there exists $C>0$, such that
\beq\label{n1}
|f(\xi+i\eta)|\leq C e^{M(k|\eta|)},\, \forall\xi\in K, \forall\eta\in\RR^d.
\eeq
Then, there exists $S\in\DD'^{*}(\RR^d_x)$ such that $e^{-x\xi}S(x)\in\SSS'^{*}(\RR^d_x)$, for all $\xi\in B$ and
\beq\label{n4}
\mathcal{L}(S)(\xi+i\eta)=\mathcal{F}_{x\rightarrow\eta}\left(e^{-x\xi}S(x)\right)(\xi+i\eta)=f(\xi+i\eta),\,\, \xi\in B,\, \eta\in\RR^d.
\eeq
\end{theorem}

\begin{proof}
Because of (\ref{n1}), for every fixed $\xi\in B$, $f_{\xi}=f(\xi+i\eta)\in\SSS'^{*}(\RR^d_{\eta})$. Put $T_{\xi}(x)=\mathcal{F}^{-1}_{\eta\rightarrow x}\left(f_{\xi}(\eta)\right)(x)\in\SSS'^{*}(\RR^d_x)$ and $S_{\xi}(x)=e^{x\xi}T_{\xi}(x)\in\DD'^{*}(\RR^d_x)$. We will show that $S_{\xi}$ does not depend on $\xi\in B$. Let $U$ be an arbitrary, but fixed, bounded connected open subset of $B$, such that $K=\overline{U}\subset\subset B$.\\
\indent Let $c>2$ be such that $|\xi_j|\leq c/2$, for $\xi=(\xi_1,...,\xi_d)\in K$. In the $(M_p)$ case, choose $s>0$ such that $\ds\int_{\RR^d}e^{M(k|\eta|)}e^{-M(\frac{s}{2}|\eta|)}d\eta<\infty$ and $e^{2M(k|\eta|)}\leq \tilde{c} e^{M(\frac{s}{2}|\eta|)}$, for some constant $\tilde{c}>0$. For the $\{M_p\}$ case, by the conditions in the theorem, for every $k>0$ there exists $C>0$, such that $\ln_+|f(\xi+i\eta)|\leq M(k|\eta|)+\ln C$ for all $\xi\in K$ and $\eta\in\RR^d$. The same estimate holds for the nonnegative increasing function
\beqs
g(\rho)=\sup_{|\eta|\leq\rho}\sup_{\xi\in K}\ln_+|f(\xi+i\eta)|.
\eeqs
If we use lemma \ref{nnl} for this function we get that there exists subordinate function $\epsilon(\rho)$ and a constant $C>1$ such that $g(\rho)\leq M(\epsilon(\rho))+\ln C$. From this we have that $\ln_+|f(\xi+i\eta)|\leq g(|\eta|)\leq M(\epsilon(|\eta|))+\ln C$, i.e.
\beq\label{n2}
|f(\xi+i\eta)|\leq Ce^{M(\epsilon(|\eta|))},\, \forall\xi\in K, \forall\eta\in\RR^d,
\eeq
for some $C>1$. By lemma 3.12 of \cite{Komatsu1}, there exists another sequence $\tilde{N}_p$, which satisfies $(M.1)$, such that $\tilde{N}(\rho)\geq M(\epsilon(\rho))$ and $k'_p=\tilde{n}_p/m_p\longrightarrow\infty$ when $p\longrightarrow\infty$. Take $(k_p)\in\mathfrak{R}$ such that $k_p\leq k'_p$, $p\in\ZZ_+$. Then
\beqs
e^{N_{k_p}(\rho)}=\sup_p\frac{\rho^p}{M_p\prod_{j=1}^p k_j}\geq \sup_p\frac{\rho^p}{M_p\prod_{j=1}^p k'_j}=e^{\tilde{N}(\rho)}\geq e^{M(\epsilon(\rho))}.
\eeqs
Hence, from (\ref{n2}), it follows that $|f(\xi+i\eta)|\leq Ce^{N_{k_p}(|\eta|)}$, for all $\xi\in K$ and $\eta\in\RR^d$. Choose $(s_p)\in\mathfrak{R}$ such that $\ds\int_{\RR^d}e^{N_{k_p}(|\eta|)}e^{-N_{2s_p}(|\eta|)}d\eta<\infty$ and $e^{2N_{k_p}(|\eta|)}\leq \tilde{c} e^{N_{2s_p}(|\eta|)}$, for some $\tilde{c}>0$.\\
\indent Now, for the chosen $c$ and $s$, resp. $(s_p)$, by the discussion before the theorem, we can find $l>0$, resp. $(l_p)\in\mathfrak{R}$, and entire functions $P_l(w)$ as in (\ref{u1}), resp. $P_{l_p}(w)$ as in (\ref{u2}), such that they don't have zeroes in $W=\RR^d+i\{v\in\RR^d||v_j|\leq c,\,j=1,...,d\}$ and the following estimates hold
\beqs
\left|\partial^{\alpha}_x\frac{1}{P_l(x)}\right|\leq C\cdot\frac{\alpha!}{r^{|\alpha|}}e^{-M\left(s|x|\right)},\, \mbox{resp.}\, \left|\partial^{\alpha}_x\frac{1}{P_{l_p}(x)}\right|\leq C\cdot\frac{\alpha!}{r^{|\alpha|}}e^{-N_{s_p}(|x|)},\, x\in\RR^d,\alpha\in\NN^d,
\eeqs
where $C$ depends on $s$ and $l$, resp. $(s_p)$ and $(l_p)$, and $M_p$; $r\leq c$ is arbitrary but fixed. For shorter notation, we will denote $P_l(w)$ and $P_{l_p}(w)$ by $P(w)$ in both cases. Define the entire functions $\ds P_{\xi}(w)=P(w-i\xi)=\prod_{j=q}^{\infty}\left(1+\frac{(w-i\xi)^2}{l^2 m_j^2}\right)$ in the $(M_p)$ case, resp. $\ds P_{\xi}(w)=P(w-i\xi)=\prod_{j=q}^{\infty}\left(1+\frac{(w-i\xi)^2}{l_j^2 m_j^2}\right)$ in the $\{M_p\}$ case. As we noted in the construction of the entire functions $P(w)$ (the discussion before the theorem), $P(w)$ satisfies the equivalent conditions a) and b) of proposition 4.5 of \cite{Komatsu1}. Hence, there exist $L>0$ and $C'>0$, resp. for every $L>0$ there exists $C'>0$, such that $|P(w)|\leq C'e^{M(L|w|)}$, $w\in\CC^d$ and $P(D)$ are ultradifferential operators of $(M_p)$, resp. $\{M_p\}$, type. So, we obtain
\beqs
|P_{\xi}(w)|=|P(w-i\xi)|\leq C'e^{M(L|w-i\xi|)}\leq C''e^{M(2L|w|)},\, w\in\CC^d,
\eeqs
because $\xi=(\xi_1,...,\xi_d)$ is such that $|\xi_j|\leq c/2$, for $j=1,...,d$. Hence, by proposition 4.5 of \cite{Komatsu1}, $P_{\xi}(D)$ is an ultradifferential operator of class $(M_p)$, resp. of class $\{M_p\}$, for every $\xi=(\xi_1,...,\xi_d)$ such that $|\xi_j|\leq c/2$, $j=1,...,d$. Moreover, by the properties of $P(w)$, it follows that $P_{\xi}(w)$ is an entire function that doesn't have zeroes in $\RR^d+i\{v\in\RR^d||v_j|\leq c/2,\,j=1,...,d\}$ for all $\xi\in K$. So, by using the Cauchy integral formula to estimate the derivatives, one obtains that $P_{\xi}(\eta)$ and $1/P_{\xi}(\eta)$ are multipliers for $\SSS'^{*}(\RR^d_{\eta})$. Also, by (\ref{uu1}), we have $|P_{\xi}(\eta)|=|P(\eta-i\xi)|\geq \tilde{C}e^{M(s|\eta-i\xi|)}\geq \tilde{C}'e^{M(\frac{s}{2}|\eta|)}$, for all $\xi\in K$ and $\eta\in\RR^d$ in the $(M_p)$ case and similarly, $|P_{\xi}(\eta)|=|P(\eta-i\xi)|\geq \tilde{C}e^{N_{s_p}(|\eta-i\xi|)}\geq \tilde{C}'e^{N_{2s_p}(|\eta|)}$, for all $\xi\in K$ and $\eta\in\RR^d$, in the $\{M_p\}$ case. For $\xi\in B$, put $f_{\xi}(\eta)=f(\xi+i\eta)$. Then $f_{\xi}(\eta)/P_{\xi}(\eta)\in L^1\left(\RR^d_{\eta}\right)\cap \EE^*\left(\RR^d_{\eta}\right)$, for all $\xi\in K$. Observe that
\beqs
e^{x\xi}\mathcal{F}^{-1}_{\eta\rightarrow x}\left(f_{\xi}(\eta)\right)(x)=e^{x\xi}\mathcal{F}^{-1}_{\eta\rightarrow x}\left(\frac{f_{\xi}(\eta)P_{\xi}(\eta)}{P_{\xi}(\eta)}\right)(x)=
e^{x\xi}P_{\xi}(D_x)\left(\mathcal{F}^{-1}_{\eta\rightarrow x}\left(\frac{f_{\xi}(\eta)}{P_{\xi}(\eta)}\right)(x)\right),
\eeqs
i.e.
\beq\label{n3}
S_{\xi}(x)=e^{x\xi}P_{\xi}(D_x)\left(\mathcal{F}^{-1}_{\eta\rightarrow x}\left(\frac{f_{\xi}(\eta)}{P_{\xi}(\eta)}\right)(x)\right).
\eeq
Let $\ds P(w)=\sum_{\alpha}c_{\alpha}w^{\alpha}$. For simpler notation, put $R(\eta)=f_{\xi}(\eta)/P_{\xi}(\eta)$ and calculate as follows
\beqs
P(D_x)\left(e^{x\xi}\mathcal{F}^{-1}_{\eta\rightarrow x}(R)(x)\right)&=&\sum_{\alpha} c_{\alpha}
\sum_{\beta\leq\alpha}{\alpha\choose\beta}(-i\xi)^{\beta}e^{x\xi}D^{\alpha-\beta}_x
\mathcal{F}^{-1}_{\eta\rightarrow x}(R)(x)\\
&=&e^{x\xi}\sum_{\alpha} c_{\alpha}
\sum_{\beta\leq\alpha}{\alpha\choose\beta}(-i\xi)^{\beta}D^{\alpha-\beta}_x
\mathcal{F}^{-1}_{\eta\rightarrow x}(R)(x).
\eeqs
Note that\\
$\ds\sum_{\alpha} c_{\alpha}
\sum_{\beta\leq\alpha}{\alpha\choose\beta}(-i\xi)^{\beta}D^{\alpha-\beta}_x
\mathcal{F}^{-1}_{\eta\rightarrow x}(R)(x)$
\beqs
&=&\mathcal{F}^{-1}_{\eta\rightarrow x}\left(\sum_{\alpha} c_{\alpha}
\sum_{\beta\leq\alpha}{\alpha\choose\beta}(-i\xi)^{\beta}\eta^{\alpha-\beta}R(\eta)\right)(x)
=\mathcal{F}^{-1}_{\eta\rightarrow x}\left(\sum_{\alpha} c_{\alpha}(\eta-i\xi)^{\alpha}R(\eta)\right)(x)\\
&=&\mathcal{F}^{-1}_{\eta\rightarrow x}\left(P(\eta-i\xi)R(\eta)\right)(x)=\mathcal{F}^{-1}_{\eta\rightarrow x}\left(P_{\xi}(\eta)R(\eta)\right)(x)=P_{\xi}(D_x)\mathcal{F}^{-1}_{\eta\rightarrow x}(R)(x).
\eeqs
From this and (\ref{n3}), we get $\ds S_{\xi}(x)=P(D_x)\left(e^{x\xi}\mathcal{F}^{-1}_{\eta\rightarrow x}\left(\frac{f_{\xi}(\eta)}{P_{\xi}(\eta)}\right)(x)\right)$. Now, for $w=\eta-i\xi$, we have
\beqs
e^{x\xi}\mathcal{F}^{-1}_{\eta\rightarrow x}\left(\frac{f_{\xi}(\eta)}{P_{\xi}(\eta)}\right)(x)
=\frac{1}{(2\pi)^d}\int_{\RR^d}\frac{f(\xi+i\eta)e^{(\xi+i\eta)x}}{P(\eta-i\xi)}d\eta
=\frac{1}{(2\pi)^d}\int_{\RR^d-i\xi}\frac{f(iw)e^{iwx}}{P(w)}dw.
\eeqs
The function $\ds\frac{f(iw)e^{iwx}}{P(w)}$ is analytic for $iw\in U+i\RR^d$, i.e. $w\in\RR^d-iU$ (because $P(w)$ is analytic in the last set and doesn't have zeroes there). Using the growth estimates for $f$ and $P$, from the theorem of Cauchy-Poincar\'e, it follows that the last integral doesn't depend on $\xi\in U$. From this and the arbitrariness of $U$ it follows that $S_{\xi}(x)$ doesn't depend on $\xi\in B$. We will denote this by $S(x)$. Now, by the observations in the beginning, it follows that $\mathcal{F}_{x\rightarrow\eta}\left(e^{-x\xi}S(x)\right)=f_{\xi}$ as ultradistributions in $\eta$ for every fixed $\xi\in B$. By theorem \ref{t1}, it follows that $\mathcal{F}_{x\rightarrow\eta}\left(e^{-x\xi}S(x)\right)$ is analytic function for $\zeta=\xi+i\eta\in B+i\RR^d$, hence the equality (\ref{n4}) holds pointwise.
\end{proof}

\begin{remark}
If $f$ is an analytic function on $O=B+i\RR^d_{\eta}$ and satisfies the conditions of the previous theorem then, by this theorem and theorem \ref{t1}, it follows that $f$ is analytic on $\mathrm{ch\,}B+i\RR^d_{\eta}$ and satisfies the estimates (\ref{3}) for every $K\subset\subset \mathrm{ch\,}B$.
\end{remark}



\end{document}